\documentclass{article}%
\usepackage{amsfonts}
\usepackage{amsmath}
\usepackage{amssymb}
\usepackage{graphicx}%
\newtheorem{theorem}{Theorem}

\newtheorem{corollary}[theorem]{Corollary}

\newtheorem{proposition}[theorem]{Proposition}
\newtheorem{remark}[theorem]{Remark}

\newenvironment{proof}[1][Proof]{\noindent\textbf{#1.} }{\ \rule{0.5em}{0.5em}}
\begin{document}

\title{Parallelism between locally conformal symplectic manifolds and contact manifolds}
\author{Eug\`{e}ne Okassa\\Universit\'{e} Marien Ngouabi\\Facult\'{e} des Sciences et Techniques\\D\'{e}partement de Math\'{e}matiques\\B.P.69 - Brazzaville - (Congo)\\e-mail:eugeneokassa@yahoo.fr}
\date{}
\maketitle

\begin{abstract}
We give the parallelism between locally conformal symplectic manifolds and
contact manifolds. We also give the generalization of exact contact manifolds.

\end{abstract}

\bigskip



\section{Introduction}

Let $A$ be a commutative algebra with unit $1_{A}$ over a commutative field
$K$ with characteristic $zero$, and Diff$_{K}(A)$, the Lie algebra of
differential operators of order $\leq1$ on $A$.

We recall that a Lie-Rinehart algebra is a pair $(\mathcal{G},\rho)$ where
$\mathcal{G}$ is simultaneously an $A$-module and a $K$-Lie algebra,  which
Lie algebra bracket $\left[  ,\right]  $, and
\[
\rho:\mathcal{G}\longrightarrow\text{Diff}_{K}(A)
\]
is simultaneously a morphism of $A$-modules and $K$-Lie algebras satisfying
\[
\left[  x,a\cdot y\right]  =\left[  \rho(x)(a)-a\cdot\rho(x)(1_{A})\right]
\cdot y+a\cdot\left[  x,y\right]
\]
for any $a\in A$ and $x$ , $y$ $\in$ $\mathcal{G}$ $\cite{oka1}$.

Let $(\mathcal{G},\rho)$ be a Lie-Rinehart algebra and
\[
\mathfrak{L}_{sks}(\mathcal{G},A)=\bigoplus\limits_{p\in\mathbb{N}%
}\mathfrak{L}_{sks}^{p}(\mathcal{G},A)
\]
where $\mathfrak{L}_{sks}^{p}(\mathcal{G},A)$ is the module of skew-symmetric
$A$-multilinear maps of degree $p$ from $\mathcal{G}$ to $A$ and finally
\[
d_{\rho}:\mathfrak{L}_{sks}(\mathcal{G},A)\longrightarrow\mathfrak{L}%
_{sks}(\mathcal{G},A)
\]
the cohomology operator associated with the representation $\rho$.

We recall that the pair $(\mathfrak{L}_{sks}(\mathcal{G},A),d_{\rho})$ is a
differential algebra \cite{oka1}.

For any $x\in\mathcal{G}$, the map
\[
i_{x}:\mathfrak{L}_{sks}(\mathcal{G},A)\longrightarrow\mathfrak{L}%
_{sks}(\mathcal{G},A)
\]
defined by
\[
(i_{x}f)(x_{1},x_{2},...,x_{p-1})=f(x,x_{1},x_{2},...,x_{p-1})\text{,}%
\]
for $x_{1},x_{2},...,x_{p-1}$ elements of $\mathcal{G}$ and for any
$f\in\mathfrak{L}_{sks}^{p}(\mathcal{G},A)$, is a derivation of degree $-1$
\cite{god}. The map
\[
\theta_{x}=\left[  i_{x},d_{\rho}\right]  =i_{x}\circ d_{\rho}+d_{\rho}\circ
i_{x}:\mathfrak{L}_{sks}(\mathcal{G},A)\longrightarrow\mathfrak{L}%
_{sks}(\mathcal{G},A)
\]
is a differential operator of order $\leq1$ and of degree $zero$ satisfying,
for any $y$ $\in$ $\mathcal{G}$, $a\in A$,%

\begin{align*}
\left[  \theta_{x},i_{y}\right]   &  =i_{\left[  x,y\right]  }\text{;}\\
\text{ }\theta_{x}\circ d_{\rho}  &  =d_{\rho}\circ\theta_{x}\text{;}\\
\left[  \theta_{x},\theta_{y}\right]   &  =\theta_{\left[  x,y\right]
}\text{;}\\
\text{ }\theta_{x}a  &  =\left[  \rho(x)\right]  (a)\text{.}%
\end{align*}
For any $x\in\mathcal{G}$, the bracket that defines $\theta_{x}$ is the graded commutator.

A Lie-Rinehart-Jacobi algebra structure on a Lie-Rinehart algebra
$(\mathcal{G},\rho)$ is defined by a skew-symmetric bilinear form
\[
\mu:\mathcal{G}\times\mathcal{G}\longrightarrow A
\]
such that
\[
d_{\rho}\mu=0\text{.}%
\]

The triplet $(\mathcal{G},\rho,\mu)$ is a Lie-Rinehart-Jacobi algebra
\cite{oka1}. A Lie-Rinehart-Jacobi algebra $(\mathcal{G},\rho,\mu)$ is a
Lie-Rinehart-Poisson algebra if $\rho(x)(1_{A})=0$ for any $x\in\mathcal{G}$
\cite{oka1}.

A Lie-Rinehart-Jacobi algebra (a Lie-Rinehart-Poisson algebra respectively),
$(\mathcal{G},\rho,\mu)$, is said to be a symplectic Lie-Rinehart-Jacobi
algebra (a symplectic Lie-Rinehart-Poisson algebra respectively) if the
skew-symmetric bilinear form $\mu$ is nondegenerate \cite{oka1} i.e. the
induced map
\[
\mathcal{G}\longrightarrow\mathcal{G}^{\ast},x\longmapsto i_{x}\mu,
\]
is an isomorphism of $A$-modules where $\mathcal{G}^{\ast}$ is the $A$-module
of linear forms on $\mathcal{G}$.

We recall that if a triplet $(\mathcal{G},\rho,\mu)$ is a symplectic
Lie-Rinehart-Jacobi algebra (a symplectic Lie-Rinehart-Poisson algebra
respectively), then $A$ is a Jacobi algebra ($A$ is a Poisson algebra
respectively) \cite{oka1}.

The parallelism between symplectic manifolds and (exact) contact manifolds is
given in \cite{oka3}.

The main goal of this paper is to show the parallelism between locally
conformal symplectic manifolds and contact manifolds. We also will give the
generalization of exact contact manifolds.

In what follows, $M$ denotes a paracompact and connected smooth manifold,
$C^{\infty}(M)$ the algebra of numerical functions of classe $C^{\infty}$ on
$M$, $\mathfrak{X}(M)$ the $C^{\infty}(M)$-module of vector fields on $M$, $1$
the unit of $C^{\infty}(M)$, $\mathcal{D}(M)$ the $C^{\infty}(M)$-module of
differential operators of order $\leq1$ on $C^{\infty}(M)$ and $\delta$ the
cohomology operator associated with the identically map
\[
id:\mathcal{D}(M)\longrightarrow\mathcal{D}(M)\text{.}%
\]

The term "differential operator" will mean "differential operator of order
$\leq1$".

\section{Symplectic Lie-Rinehart-Jacobi algebra structure on $\mathfrak{X}%
(M)$}

A locally conformal symplectic structure on $M$ is a pair $(\omega,\alpha)$
made up by a closed $1$-form
\[
\alpha:\mathfrak{X}(M)\longrightarrow C^{\infty}(M)
\]
and a nondegerate skew-symmetric $2$-form
\[
\omega:\mathfrak{X}(M)\times\mathfrak{X}(M)\longrightarrow C^{\infty}(M)
\]
such that
\[
d\omega=-\alpha\Lambda\omega
\]
where $d$ is the exterior differentiation operator.

When $\alpha=0$, then $M$ is a symplectic manifold.

\bigskip

\begin{proposition}
\cite{oka2}A smooth manifold $M$ is a locally conformal symplectic manifold
($M$ is a symplectic manifold respectively) if and only if $\mathfrak{X}(M)$
admits a symplectic Lie-Rinehart-Jacobi algebra structure ($\mathfrak{X}(M)$
admits a symplectic Lie-Rinehart-Poisson algebra structure respectively).
\end{proposition}

\bigskip

\section{Symplectic Lie-Rinehart-Jacobi algebra structure on $\mathcal{D}(M)$}

\bigskip

When $\varphi\in\mathcal{D}(M)$, $f$,$g\in C^{\infty}(M)$, we recall that
\begin{align*}
\left[  \varphi,f\right]   &  =\varphi(f)-f\cdot\varphi(1),\\
\left[  f,g\right]   &  =0\text{.}%
\end{align*}

\subsection{Lie-Rinehart algebra structure on $\mathcal{D}(M)$}

\bigskip

For any linear form%
\[
\alpha:\mathcal{D}(M)\longrightarrow C^{\infty}(M)\text{,}%
\]
we verify that the map%
\[
\rho_{\alpha}:\mathcal{D}(M)\longrightarrow\mathcal{D}(M),\varphi
\longmapsto\varphi+\alpha(\varphi),
\]
is $C^{\infty}(M)$-linear.

\begin{proposition}
For any linear form
\[
\alpha:\mathcal{D}(M)\longrightarrow C^{\infty}(M),
\]
then the map
\[
\rho_{\alpha}:\mathcal{D}(M)\longrightarrow\mathcal{D}(M),\text{ }%
\varphi\longmapsto\rho_{\alpha}(\varphi),
\]
is a morphism of Lie algebras if and only if
\[
\delta\alpha=(\delta1)\Lambda\alpha\text{.}%
\]

\end{proposition}

\begin{proof}
For any $\varphi$, $\psi$ $\in\mathcal{D}(M)$, we verify that
\[
\left[  \rho_{\alpha}(\varphi),\rho_{\alpha}(\psi)\right]  -\rho_{\alpha
}\left[  \varphi,\psi\right]  =\left[  \delta\alpha-(\delta1)\Lambda
\alpha\right]  (\varphi,\psi)\text{.}%
\]
And that ends the proof.
\end{proof}

\begin{theorem}
If $M$ is a smooth manifold, then a Lie-Rinehart algebra structure on
$\mathcal{D}(M)$ is always of the form $(\mathcal{D}(M),\rho_{\alpha})$ where%
\[
\alpha:\mathcal{D}(M)\longrightarrow C^{\infty}(M)
\]
is a linear form such that
\[
\delta\alpha=(\delta1)\Lambda\alpha\text{.}%
\]

\end{theorem}

\begin{proof}
The previous proposition implies the sufficient condition. For the necessary
condition, let be given a Lie-Rinehart algebra stucture $(\mathcal{D}%
(M),\rho)$ on $\mathcal{D}(M)$. For any $\varphi\in\mathcal{D}(M)$, for any
$f\in C^{\infty}(M)$, we get
\[
\left[  \varphi,f\right]  =\left[  \rho(\varphi)\right]  (f)-f\cdot\left[
\rho(\varphi)\right]  (1)\text{.}%
\]
On the other hand , we get
\[
\left[  \varphi,f\right]  =\varphi(f)-f\cdot\varphi(1)\text{.}%
\]
We deduce that
\[
\left[  \rho(\varphi)\right]  (f)-f\cdot\left[  \rho(\varphi)\right]
(1)=\varphi(f)-f\cdot\varphi(1)\text{.}%
\]
Therefore
\[
\left[  \rho(\varphi)\right]  (f)=\varphi(f)+f\cdot(\left[  \rho
(\varphi)\right]  (1)-\varphi(1))\text{.}%
\]
The map
\[
\alpha:\mathcal{D}(M)\longrightarrow C^{\infty}(M),\varphi\longmapsto\left[
\rho(\varphi)-\varphi\right]  (1),
\]
is a $C^{\infty}(M)$-linear. Thus
\[
\left[  \rho(\varphi)\right]  (f)=\varphi(f)+f\cdot\alpha(\varphi)\text{.}%
\]
We have
\[
\rho(\varphi)=\varphi+\alpha(\varphi).
\]
We finally conclude that $\rho=\rho_{\alpha}$. As $\rho$ has to be a Lie
algebras morphism, we deduce that the linear form $\alpha$ is such that
$\delta\alpha=(\delta1)\Lambda\alpha$.
\end{proof}

\bigskip

\begin{proposition}
A linear form%
\[
\alpha:\mathcal{D}(M)\longrightarrow C^{\infty}(M)
\]
satisfies%
\[
\delta\alpha=(\delta1)\Lambda\alpha
\]
if and only if%
\[
\alpha(1)\in%
\mathbb{R}
\text{ and }d(\alpha_{\left\vert \mathfrak{X}(M)\right.  })=0\text{.}%
\]

\end{proposition}

\begin{proof}
For any $\varphi=\varphi(1)+X$, $\psi=\psi(1)+Y$ two elements of
$\mathcal{D}(M)$ with $X$, $Y\in\mathfrak{X}(M)$, we have%
\begin{align*}
&  \left[  \delta\alpha-(\delta1)\Lambda\alpha\right]  (\varphi,\psi)\\
&  =\psi(1)\cdot X\left[  \alpha(1)\right]  -\varphi(1)\cdot Y\left[
\alpha(1)\right]  +\left[  d(\alpha_{\left\vert \mathfrak{X}(M)\right.
})\right]  (X,Y)\text{.}%
\end{align*}

$/\Longrightarrow$ As $\delta\alpha-(\delta1)\Lambda\alpha=0$, we have
\begin{align*}
0  &  =\left[  \delta\alpha-(\delta1)\Lambda\alpha\right]  (X,1)\\
&  =X\left[  \alpha(1)\right] \\
&  =(d\left[  \alpha(1)\right]  )(X)\text{.}%
\end{align*}

As $X$ is arbitrary, we deduce that $d\left[  \alpha(1)\right]  =0$. Thus
$\alpha(1)\in%
\mathbb{R}
$ since $M$ is connected.

We also have $d(\alpha_{\left\vert \mathfrak{X}(M)\right.  })=0$.

$\Longleftarrow/$ If $\alpha(1)\in%
\mathbb{R}
$ and $d(\alpha_{\left\vert \mathfrak{X}(M)\right.  })=0$, we immediately have
$\delta\alpha=(\delta1)\Lambda\alpha$.
\end{proof}

\bigskip

\subsection{Symplectic Lie-Rinehart-Jacobi algebra structure on $\mathcal{D}%
(M)$}

Let $(\mathcal{D}(M),\rho_{\alpha})$ be a Lie-Rinehart algebra structure on
$\mathcal{D}(M)$. The linear form
\[
\alpha:\mathcal{D}(M)\longrightarrow C^{\infty}(M)
\]
is such that $\delta\alpha=(\delta1)\Lambda\alpha$. In this case, we denote
$\delta_{\alpha}$ the cohomology operator associated with the representation
$\rho_{\alpha}$.

\begin{proposition}
For any $\eta\in\mathfrak{L}_{sks}(\mathcal{D}(M),C^{\infty}(M))$, then%
\[
\delta_{\alpha}\eta=\delta\eta+\alpha\Lambda\eta\text{.}%
\]

\end{proposition}

\begin{proof}
For any $\eta\in\mathfrak{L}_{sks}^{p}(\mathcal{D}(M),C^{\infty}(M))$ and for
any $\varphi_{1},...,\varphi_{p+1}\in\mathcal{D}(M)$, we have%
\begin{align*}
(\delta_{\alpha}\eta)(\varphi_{1},...,\varphi_{p+1})  &  =%
{\displaystyle\sum\limits_{i=1}^{p+1}}
(-1)^{i-1}\rho_{\alpha}(\varphi_{i})\left[  \eta(\varphi_{1},...,\widehat
{\varphi_{i}},...,\varphi_{p+1})\right] \\
&  +%
{\displaystyle\sum\limits_{1\leq i<j\leq p+1}}
(-1)^{i+j}\eta(\left[  \varphi_{i},\varphi_{j}\right]  ,\varphi_{1}%
,...,\widehat{\varphi_{i}},...,\widehat{\varphi_{j}},...,\varphi_{p+1})\\
&  =%
{\displaystyle\sum\limits_{i=1}^{p+1}}
(-1)^{i-1}\varphi_{i}\left[  \eta(\varphi_{1},...,\widehat{\varphi_{i}%
},...,\varphi_{p+1})\right] \\
&  +%
{\displaystyle\sum\limits_{i=1}^{p+1}}
(-1)^{i-1}\left[  \eta(\varphi_{1},...,\widehat{\varphi_{i}},...,\varphi
_{p+1})\right]  \cdot\alpha(\varphi_{i})\\
&  +%
{\displaystyle\sum\limits_{1\leq i<j\leq p+1}}
(-1)^{i+j}\eta(\left[  \varphi_{i},\varphi_{j}\right]  ,\varphi_{1}%
,...,\widehat{\varphi_{i}},...,\widehat{\varphi_{j}},...,\varphi_{p+1})\\
&  =(\delta\eta+\alpha\Lambda\eta)(\varphi_{1},...,\varphi_{p+1})\text{.}%
\end{align*}

That ends the proof.
\end{proof}

The characterization of symplectic Lie-Rinehart-Jacobi algebra structure on
$\mathcal{D}(M)$ is the following one:

\begin{proposition}
The $C^{\infty}(M)$-module $\mathcal{D}(M)$ admits a symplectic
Lie-Rinehart-Jacobi algebra structure if and only if there exists a
$C^{\infty}(M)$-linear form%
\[
\alpha:\mathcal{D}(M)\longrightarrow C^{\infty}(M)
\]
and a nondegenerate skew-symmetric bilinear form%
\[
\omega:\mathcal{D}(M)\times\mathcal{D}(M)\longrightarrow C^{\infty}(M)
\]
such that
\end{proposition}

\begin{enumerate}
\item $\delta\alpha=(\delta1)\Lambda\alpha$

\item $\delta\omega=-\alpha\Lambda\omega$.
\end{enumerate}

\begin{proof}
It is obvious.
\end{proof}

\subsection{Structure of contact manifold on $M$ when $\mathcal{D}(M)$ admits
a symplectic Lie-Rinehart-Jacobi algebra structure}

In this part, we consider a symplectic Lie-Rinehart-Jacobi algebra structure
on $\mathcal{D}(M)$ with a linear form%
\[
\alpha:\mathcal{D}(M)\longrightarrow C^{\infty}(M)
\]
such that
\[
\delta\alpha=(\delta1)\Lambda\alpha
\]
and a nondegenerate skew-symmetric bilinear form%
\[
\omega:\mathcal{D}(M)\times\mathcal{D}(M)\longrightarrow C^{\infty}(M)
\]
such that
\[
\delta\omega=-\alpha\Lambda\omega\text{.}%
\]

In this case if $\mathcal{D}(M)^{\ast}$ denotes the dual of the $C^{\infty
}(M)$-module $\mathcal{D}(M)$, the map%
\[
\mathcal{D}(M)\longrightarrow\mathcal{D}(M)^{\ast},\varphi\longmapsto
i_{\varphi}\omega,
\]
is an isomorphism of $C^{\infty}(M)$-modules.

\begin{proposition}
There exists an unique vector field $H$ on $M$ such that
\[
i_{H}\omega=-\delta1\text{.}%
\]
Moreover the linear form
\[
i_{1}\omega:\mathcal{D}(M)\longrightarrow C^{\infty}(M),\varphi\longmapsto
\omega(1,\varphi),
\]
is such that
\[
(i_{1}\omega)(H)=1\text{.}%
\]

\end{proposition}

\begin{proof}
As
\[
\omega:\mathcal{D}(M)\times\mathcal{D}(M)\longrightarrow C^{\infty}(M)
\]
is nondegenerate, let $H\in\mathcal{D}(M)$ be the unique differential operator
such that $i_{H}\omega=-\delta1$. We have $H(1)=0$. Thus $H$ is a vector field.

We deduce that%
\[
(i_{1}\omega)(H)=1\text{.}%
\]

And that ends the proof.
\end{proof}

\begin{proposition}
We get%
\[
\mathfrak{X}(M)=Ker\left[  i_{1}\omega|_{\mathfrak{X}(M)}\right]  \oplus
C^{\infty}(M)\cdot H\text{.}%
\]

\end{proposition}

\begin{proof}
For any $X\in\mathfrak{X}(M)$, we write
\[
X=\left[  X-(i_{1}\omega)(X)\cdot H\right]  +(i_{1}\omega)(X)\cdot H\text{.}%
\]
We verify that
\[
\left[  X-(i_{1}\omega)(X)\cdot H\right]  \in Ker\left[  i_{1}\omega
|_{\mathfrak{X}(M)}\right]
\]
and%
\[
Ker\left[  i_{1}\omega|_{\mathfrak{X}(M)}\right]  \cap C^{\infty}(M)\cdot
H=\left\{  0\right\}  .
\]
Thus%
\[
\mathfrak{X}(M)=Ker\left[  i_{1}\omega|_{\mathfrak{X}(M)}\right]  \oplus
C^{\infty}(M)\cdot H\text{.}%
\]

That ends the proof.
\end{proof}

The sets
\[
\mathcal{D}(M)_{H}^{\ast}=\left\{  \eta\in\mathcal{D}(M)^{\ast}/\eta
(H)=0\right\}
\]
and%
\[
\mathcal{D}(M)_{C^{\infty}(M),H}^{\ast}=\left\{  \eta\in\mathcal{D}(M)^{\ast
}/\eta|_{C^{\infty}(M)}=0;\eta(H)=0\right\}
\]
are modules over $C^{\infty}(M)$.

For any $X\in\mathfrak{X}(M)$ ( $X\in$ $Ker\left[  i_{1}\omega|_{\mathfrak{X}%
(M)}\right]  $ respectively), we verify that $i_{X}\omega\in\mathcal{D}%
(M)_{H}^{\ast}$ \ ( $i_{X}\omega\in\mathcal{D}(M)_{C^{\infty}(M),H}^{\ast}$ respectively).

\begin{proposition}
The following maps%
\[
\mathfrak{X}(M)\longrightarrow\mathcal{D}(M)_{H}^{\ast},X\longmapsto
i_{X}\omega,
\]
and%
\[
Ker\left[  i_{1}\omega|_{\mathfrak{X}(M)}\right]  \longrightarrow
\mathcal{D}(M)_{C^{\infty}(M),H}^{\ast},X\longmapsto i_{X}\omega,
\]
are isomorphisms of $C^{\infty}(M)$-modules.
\end{proposition}

\begin{proof}
Since the map%
\[
\mathcal{D}(M)\longrightarrow\mathcal{D}(M)^{\ast},\varphi\longmapsto
i_{\varphi}\omega,
\]
is an isomorphism of $C^{\infty}(M)$-modules, then the maps
\[
\mathfrak{X}(M)\longrightarrow\mathcal{D}(M)_{H}^{\ast},X\longmapsto
i_{X}\omega,
\]
and
\[
Ker\left[  i_{1}\omega|_{\mathfrak{X}(M)}\right]  \longrightarrow
\mathcal{D}(M)_{C^{\infty}(M),H}^{\ast},X\longmapsto i_{X}\omega,
\]
are injective.

Let $\eta\in\mathcal{D}(M)_{H}^{\ast}$ be a linear form on $\mathcal{D}(M)$
such that $\eta(H)=0$ and let $\varphi$ be the unique element of
$\mathcal{D}(M)$ such that%
\[
i_{\varphi}\omega=\eta\text{.}%
\]
We get%
\begin{align*}
0  &  =\eta(H)\\
&  =i_{\varphi}\omega(H)\\
&  =-(i_{H}\omega)(\varphi)\\
&  =(\delta1)(\varphi)\\
&  =\varphi(1)\text{.}%
\end{align*}
We deduce that $\varphi\in\mathfrak{X}(M)$. Thus the map%
\[
\mathfrak{X}(M)\longrightarrow\mathcal{D}(M)_{H}^{\ast},X\longmapsto
i_{X}\omega,
\]
is also surjective.

Let $\sigma\in\mathcal{D}(M)_{C^{\infty}(M),H}^{\ast}$ be a linear form on
$\mathcal{D}(M)$ such that $\sigma|_{C^{\infty}(M)}=0$ and $\sigma(H)=0$, and
let $\varphi$ be the unique element of $\mathcal{D}(M)$ such that%
\[
i_{\varphi}\omega=\sigma\text{.}%
\]
As $\sigma(H)=0$, then $\varphi\in\mathfrak{X}(M)$.

Since $\sigma|_{C^{\infty}(M)}=0$, we obtain%
\begin{align*}
0  &  =\sigma(1)\\
&  =(i_{\varphi}\omega)(1)\\
&  =-\left[  i_{1}\omega\right]  (\varphi).
\end{align*}
We deduce that $\varphi\in Ker\left[  i_{1}\omega|_{\mathfrak{X}(M)}\right]
$. Thus the map%
\[
Ker\left[  i_{1}\omega|_{\mathfrak{X}(M)}\right]  \longrightarrow
\mathcal{D}(M)_{C^{\infty}(M),H}^{\ast},X\longmapsto i_{X}\omega,
\]
is also surjective.
\end{proof}

\bigskip

\begin{corollary}
The restriction%
\[
\omega|_{Ker\left[  i_{1}\omega|_{\mathfrak{X}(M)}\right]  \times Ker\left[
i_{1}\omega|_{\mathfrak{X}(M)}\right]  }:Ker\left[  i_{1}\omega|_{\mathfrak{X}%
(M)}\right]  \times Ker\left[  i_{1}\omega|_{\mathfrak{X}(M)}\right]
\longrightarrow C^{\infty}(M)
\]
is a nondegenerate skew-symmetric bilinear form on $Ker\left[  i_{1}%
\omega|_{\mathfrak{X}(M)}\right]  $.
\end{corollary}

\begin{proof}
It is obvious since the map
\[
Ker\left[  i_{1}\omega|_{\mathfrak{X}(M)}\right]  \longrightarrow Ker\left[
i_{1}\omega|_{\mathfrak{X}(M)}\right]  ^{\ast},X\longmapsto i_{X}%
\omega|_{Ker\left[  i_{1}\omega|_{\mathfrak{X}(M)}\right]  },
\]

is an isormorphism of $C^{\infty}(M)$-modules.
\end{proof}

\bigskip For any $f\in C^{\infty}(M)$, the linear form
\[
\delta_{\alpha}f-\left[  H(f)+f\cdot\alpha(H)\right]  \cdot i_{1}\omega
-f\cdot\left[  1+\alpha(1)\right]  \cdot\delta1:\mathcal{D}(M)\longrightarrow
C^{\infty}(M),
\]
belongs to $\mathcal{D}(M)_{C^{\infty}(M),H}^{\ast}$. We denote $\varphi_{f}$
the unique element of $\mathcal{D}(M)$ such that
\[
i_{\varphi_{f}}\omega=\delta_{\alpha}f
\]
and $X_{f}$ the unique element of $Ker\left[  i_{1}\omega|_{\mathfrak{X}%
(M)}\right]  $ such that
\[
i_{X_{f}}\omega=\delta_{\alpha}f-\left[  H(f)+f\cdot\alpha(H)\right]  \cdot
i_{1}\omega-f\cdot\left[  1+\alpha(1)\right]  \cdot\delta1\text{.}%
\]
For any $f,g\in C^{\infty}(M)$, the bracket
\[
\left\{  f,g\right\}  =-\omega(\varphi_{f},\varphi_{g})
\]
is a Jacobi bracket on $C^{\infty}(M)$. Thus $M$ is a Jacobi manifold
\cite{oka1}.

We verify that
\[
\varphi_{f}=\left[  \rho_{\alpha}(H)\right]  (f)+X_{f}-f\cdot\left[
1+\alpha(1)\right]  \cdot H\text{.}%
\]
If we denote $H_{\alpha}=\left[  1+\alpha(1)\right]  \cdot\rho_{\alpha}(H)$,
then we have%
\[
\left\{  f,g\right\}  =-\omega(X_{f},X_{g})-f\cdot H_{\alpha}(g)+g\cdot
H_{\alpha}(f)\text{.}%
\]

\begin{remark}
We recall that as $\omega$ is a nondegenerate skew-symmetric bilinear form on
$\mathcal{D}(M)$, then the dimension of $M$ is odd \cite{oka3}.
\end{remark}

\begin{theorem}
\bigskip If the dimension of $M$ is $2n+1$, then the differential form%
\[
\left[  i_{1}\omega|_{\mathfrak{X}(M)}\right]  \Lambda\left[  \omega
|_{\mathfrak{X}(M)\times\mathfrak{X}(M)}\right]  ^{n}%
\]
is a volume form on $M$.
\end{theorem}

\begin{proof}
For any $x\in M$ we have $H(x)\neq0$ since $(i_{1}\omega)(H)=1$. Thus the
$1$-form $\left[  i_{1}\omega|_{\mathfrak{X}(M)}\right]  $ is nonzero
everywhere. Let $x\in M$ and let $T_{x}M$ be the tangent vector space at $x$.
As the dimension of $M$ is odd, let $2n+1$ be the dimension of $M$. The set%
\[
(Ker\left[  i_{1}\omega|_{\mathfrak{X}(M)}\right]  )_{x}=\left\{  X(x)\in
T_{x}M/X\in Ker\left[  i_{1}\omega|_{\mathfrak{X}(M)}\right]  \right\}
\]
is a vector space of dimension $2n$. Since
\[
\omega|_{Ker\left[  i_{1}\omega|_{\mathfrak{X}(M)}\right]  \times Ker\left[
i_{1}\omega|_{\mathfrak{X}(M)}\right]  }:Ker\left[  i_{1}\omega|_{\mathfrak{X}%
(M)}\right]  \times Ker\left[  i_{1}\omega|_{\mathfrak{X}(M)}\right]
\longrightarrow C^{\infty}(M)
\]
is a nondegenerate skew-symmetric bilinear form on the $C^{\infty}(M)$-module%
\[
Ker\left[  i_{1}\omega|_{\mathfrak{X}(M)}\right]  \text{,}%
\]
then $(\omega|_{Ker\left[  i_{1}\omega|_{\mathfrak{X}(M)}\right]  \times
Ker\left[  i_{1}\omega|_{\mathfrak{X}(M)}\right]  })(x)$ is a nondegenerate
skew-symmetric bilinear form on the vector space $(Ker\left[  i_{1}%
\omega|_{\mathfrak{X}(M)}\right]  )_{x}$ .

We deduce that $Ker\left[  i_{1}\omega|_{\mathfrak{X}(M)}\right]  )_{x}$ is a
symplectic vector space and
\[
(\omega|_{Ker\left[  i_{1}\omega|_{\mathfrak{X}(M)}\right]  \times Ker\left[
i_{1}\omega|_{\mathfrak{X}(M)}\right]  })^{n}(x)
\]
is a volume form. We also deduce that $(\omega|_{Ker\left[  i_{1}%
\omega|_{\mathfrak{X}(M)}\right]  \times Ker\left[  i_{1}\omega|_{\mathfrak{X}%
(M)}\right]  })^{n}(x)\neq0$. Let $(v_{1},v_{2},...,v_{2n})$ be a basis of
$(Ker\left[  i_{1}\omega|_{\mathfrak{X}(M)}\right]  )_{x}$. We have
\[
(\omega|_{Ker\left[  i_{1}\omega|_{\mathfrak{X}(M)}\right]  \times Ker\left[
i_{1}\omega|_{\mathfrak{X}(M)}\right]  })^{n}(x)\left(  v_{1},v_{2}%
,...,v_{2n}\right)  \neq0\text{.}%
\]
We note that
\[
\nu=\left[  i_{1}\omega|_{\mathfrak{X}(M)}\right]  (x)\Lambda(\omega
|_{Ker\left[  i_{1}\omega|_{\mathfrak{X}(M)}\right]  \times Ker\left[
i_{1}\omega|_{\mathfrak{X}(M)}\right]  })^{n}(x)
\]
is nonzero since
\begin{align*}
&  \nu(H(x),v_{1},v_{2},...,v_{2n})\\
&  =(\omega|_{Ker\left[  i_{1}\omega|_{\mathfrak{X}(M)}\right]  \times
Ker\left[  i_{1}\omega|_{\mathfrak{X}(M)}\right]  })^{n}(x)\left(  v_{1}%
,v_{2},...,v_{2n}\right)  \neq0\text{.}%
\end{align*}
As $x$ is abitrary, we conclude that $\left[  i_{1}\omega|_{\mathfrak{X}%
(M)}\right]  \Lambda\left[  \omega|_{\mathfrak{X}(M)\times\mathfrak{X}%
(M)}\right]  ^{n}$ is a volume form on $M$.
\end{proof}

\begin{corollary}
If $\mathcal{D}(M)$ admits a symplectic Lie-Rinehart-Jacobi algebra structure,
then $M$ is a nonexact contact manifold in the sense of Andr\'{e} Lichnerowicz.
\end{corollary}

\bigskip

In what follows, we give a generalization of exact and nonexact contact manifolds.

\begin{proposition}
We get
\[
\left[  1+\alpha(1)\right]  \cdot\omega=\delta_{\alpha}(i_{1}\omega)\text{.}%
\]

\end{proposition}

\begin{proof}
For any $\varphi,\psi\in\mathcal{D}(M)$, we have
\begin{align*}
(\delta\omega)(1,\varphi,\psi)  &  =\omega(\varphi,\psi)-\varphi\left[
(i_{1}\omega)(\psi)\right]  +\psi\left[  (i_{1}\omega)(\varphi)\right]
-\omega(\left[  1,\varphi\right]  ,\psi)\\
&  +\omega(\left[  1,\psi\right]  ,\varphi)-\omega(\left[  \varphi
,\psi\right]  ,1)\text{.}%
\end{align*}

As $\left[  1,\varphi\right]  =\left[  1,\psi\right]  =0$, we get
\begin{align*}
(\delta\omega)(1,\varphi,\psi)  &  =\omega(\varphi,\psi)-\varphi\left[
(i_{1}\omega)(\psi)\right]  +\psi\left[  (i_{1}\omega)(\varphi)\right]
+(i_{1}\omega)(\left[  \varphi,\psi\right]  )\\
&  =\left[  \omega-\delta(i_{1}\omega)\right]  (\varphi,\psi)\text{.}%
\end{align*}

On the other hand, we get%
\begin{align*}
(-\alpha\Lambda\omega)(1,\varphi,\psi))  &  =-\alpha(1)\cdot\omega
(\varphi,\psi)+\alpha(\varphi)\cdot(i_{1}\omega)(\psi)-\alpha(\psi)\cdot
(i_{1}\omega)(\varphi)\\
&  =\left[  -\alpha(1)\cdot\omega+\alpha\Lambda(i_{1}\omega)\right]
(\varphi,\psi)\text{.}%
\end{align*}

As $\delta\omega=-\alpha\Lambda\omega$, we conclude that
\[
\omega-\delta(i_{1}\omega)=-\alpha(1)\cdot\omega+\alpha\Lambda(i_{1}%
\omega)\text{.}%
\]

Thus%
\begin{align*}
\left[  1+\alpha(1)\right]  \cdot\omega &  =\delta(i_{1}\omega)+\alpha
\Lambda(i_{1}\omega)\\
&  =\delta_{\alpha}(i_{1}\omega)\text{.}%
\end{align*}

That ends the proof.
\end{proof}

As $\delta\alpha=(\delta1)\Lambda\alpha$, then $\alpha(1)\in%
\mathbb{R}
$.

If $\alpha(1)\neq-1$, we have
\[
\omega=\delta_{\alpha}\left[  i_{1}(\frac{1}{1+\alpha(1)}\cdot\omega)\right]
\text{.}%
\]

In this case, we will say that $M$ is an exact contact manifold since $\omega$
is $\delta_{\alpha}$-exact.

If $\alpha(1)=-1$, we will say that $M$ is a nonexact contact manifold.

\bigskip

\subsection{Structure of symplectic Lie-Rinehart-Jacobi algebra on
$\mathcal{D}(M)$ when $M$ is a contact manifold}

Let $M$ be a contact manifold with dimension $2n+1$. In this case, there
exists an $1$-form%
\[
\beta:\mathfrak{X}(M)\longrightarrow C^{\infty}(M)
\]
and a skew-symmetric $2$-form
\[
\Omega:\mathfrak{X}(M)\times\mathfrak{X}(M)\longrightarrow C^{\infty}(M)
\]
such that%
\[
\beta\Lambda\Omega^{n}%
\]
is a volume form on $M$.

Let $E$ be the fundamental vector field of the contact manifold $M$
\cite{lic1}. We have
\[
\beta(E)=1
\]
and%
\[
i_{E}\Omega=0\text{.}%
\]
We get%
\[
\mathfrak{X}(M)=Ker\beta\oplus C^{\infty}(M)\cdot E\text{.}%
\]
The restriction
\[
\Omega|_{Ker\beta\times Ker\beta}:Ker\beta\times Ker\beta\longrightarrow
C^{\infty}(M)
\]
is a nondegenerate skew-symmetric bilinear form on $Ker\beta$ \cite{lic1}.

Thus, we have
\[
\mathcal{D}(M)=C^{\infty}(M)\oplus Ker\beta\oplus C^{\infty}(M)\cdot E\text{.}%
\]
If
\[
\pi:\mathcal{D}(M)\longrightarrow\mathfrak{X}(M)
\]
is the canonical surjection, the linear form
\[
\widetilde{\beta}=\beta\circ\pi:\mathcal{D}(M)\longrightarrow C^{\infty}(M)
\]
is such that 
\[
\widetilde{\beta}|_{C^{\infty}(M)}=0
\]
and
\[
\widetilde{\beta}|_{\mathfrak{X}(M)}=\beta\text{.}%
\]

For any $\varphi,$ $\psi$ two elements of $\mathcal{D}(M)$, we have
\begin{align*}
\varphi &  =\varphi(1)+X+\widetilde{\beta}(\varphi)\cdot E\\
\psi &  =\psi(1)+Y+\widetilde{\beta}(\psi)\cdot E
\end{align*}
where $X,Y\in Ker\beta$. The map%
\[
\overline{\Omega}:\mathcal{D}(M)\times\mathcal{D}(M)\longrightarrow C^{\infty
}(M),(\varphi,\psi)\longmapsto\Omega(X,Y),
\]
is $C^{\infty}(M)$-bilinear and skew-symmetric.

The map
\[
\widetilde{\Omega}=\overline{\Omega}+(\delta1)\Lambda\widetilde{\beta
}:\mathcal{D}(M)\times\mathcal{D}(M)\longrightarrow C^{\infty}(M)
\]
is a skew-symmetric bilinear form.

\begin{proposition}
We get%
\begin{align*}
i_{1}\overline{\Omega}  &  =0\text{;}\\
i_{1}\widetilde{\Omega}  &  =\widetilde{\beta}\text{;}\\
i_{E}\widetilde{\Omega}  &  =-\delta1\text{.}%
\end{align*}

\end{proposition}

\begin{proof}
It is obvious.
\end{proof}

\begin{proposition}
The skew-symmetric bilinear form
\[
\widetilde{\Omega}=\overline{\Omega}+(\delta1)\Lambda\widetilde{\beta
}:\mathcal{D}(M)\times\mathcal{D}(M)\longrightarrow C^{\infty}(M)
\]
is nondegenerated.
\end{proposition}

\begin{proof}
Let $\varphi\in\mathcal{D}(M)$ such that $\widetilde{\Omega}(\varphi,\psi)=0$
for any $\psi\in\mathcal{D}(M)$. We can write $\varphi=\varphi(1)+X+\widetilde
{\beta}(\varphi)\cdot E$ with $X\in Ker\beta$.

For $\psi=1$, we get%
\begin{align*}
0  &  =\widetilde{\Omega}(\varphi,1)\\
&  =\overline{\Omega}(\varphi,1)+\varphi(1)\cdot\widetilde{\beta}%
(1)-1\cdot\widetilde{\beta}(\varphi)\\
&  =-(i_{1}\overline{\Omega})(\varphi)-\widetilde{\beta}(\varphi)\\
&  =-\widetilde{\beta}(\varphi)\text{.}%
\end{align*}

Thus $\widetilde{\beta}(\varphi)=0$.

For $\psi=E$, we get
\begin{align*}
0  &  =\widetilde{\Omega}(\varphi,E)\\
&  =-(i_{E}\widetilde{\Omega})(\varphi)\\
&  =(\delta1)(\varphi)\\
&  =\varphi(1)\text{.}%
\end{align*}

Thus $\varphi(1)=0$.

As $\widetilde{\beta}(\varphi)=0$ and $\varphi(1)=0$, we have $\varphi=X$.
Thus for any $\psi=Y\in Ker\beta$, we get
\begin{align*}
0  &  =\widetilde{\Omega}(\varphi,Y)\\
&  =\widetilde{\Omega}(X,Y)\\
&  =\Omega(X,Y)\text{.}%
\end{align*}

As
\[
\Omega|_{Ker\beta\times Ker\beta}:Ker\beta\times Ker\beta\longrightarrow
C^{\infty}(M)
\]
is nondegenerated, we deduce that $X=0$. We conclude that $\varphi=0$ and the
map
\[
\mathcal{D}(M)\longrightarrow\mathcal{D}(M)^{\ast},\varphi\longmapsto
i_{\varphi}\widetilde{\Omega},
\]
is injective.

The map
\[
\mathcal{D}(M)\longrightarrow\mathcal{D}(M)^{\ast},\varphi\longmapsto
i_{\varphi}\widetilde{\Omega},
\]
is also surjective since if
\[
\nu:\mathcal{D}(M)\longrightarrow C^{\infty}(M)
\]
is a linear form on $\mathcal{D}(M)$ and if $X$ is the unique element of
$Ker\beta$ such that $i_{X}\Omega=\nu|_{Ker\beta}$, the differential operator
\[
\varphi=\nu(E)+X-\nu(1)\cdot E
\]
is such that
\[
i_{\varphi}\widetilde{\Omega}=\nu\text{.}%
\]
Thus
\[
\widetilde{\Omega}:\mathcal{D}(M)\times\mathcal{D}(M)\longrightarrow
C^{\infty}(M)
\]
is a nondegenerate skew-symmetric bilinear form.
\end{proof}

In what follows, we give the characterization of a contact manifold in terms
of symplectic Lie-Rinehart-Jacobi algebra structure on $\mathcal{D}(M)$.

We consider the linear form
\[
\alpha=\left[  1+\alpha(1)\right]  \cdot\delta1+i_{E}\delta\widetilde{\beta
}+\alpha(E)\cdot\widetilde{\beta}:\mathcal{D}(M)\longrightarrow C^{\infty}(M)
\]
on $\mathcal{D}(M)$ with
\[
\alpha(1)\in%
\mathbb{R}
\text{ and }d\left[  \alpha(E)\cdot\beta+i_{E}d\beta\right]  =0\text{.}%
\]
In this case, we have
\[
\delta\alpha=(\delta1)\Lambda\alpha\text{.}%
\]
We have the following properties:

\begin{proposition}
We get

\begin{enumerate}
\item $\left[  1+\alpha(1)\right]  \cdot\widetilde{\Omega}=\delta
\widetilde{\beta}+\alpha\Lambda\widetilde{\beta};$

\item $\alpha(E)\cdot\widetilde{\Omega}=(\delta1)\Lambda\alpha-i_{E}%
\delta\widetilde{\Omega};$

\item $\beta\left[  X,E\right]  \cdot\widetilde{\Omega}=\alpha\Lambda
i_{X}\widetilde{\Omega}-i_{X}\delta\widetilde{\Omega},$ for any $X\in
Ker\beta$.
\end{enumerate}
\end{proposition}

\begin{proof}
For any $x\in M$, as the matrix of $\widetilde{\Omega}(x)$ is regular, then
there exists $f\in C^{\infty}(M)$, $g\in C^{\infty}(M)$ and $h_{X}\in
C^{\infty}(M)$ for any $X\in Ker\beta$ such that

$1/f\cdot\widetilde{\Omega}=\delta\widetilde{\beta}+\alpha\Lambda
\widetilde{\beta}$ ,

$2/g\cdot\widetilde{\Omega}=(\delta1)\Lambda\alpha-i_{E}\delta\widetilde
{\Omega},$

$3/h_{X}\cdot\widetilde{\Omega}=\alpha\Lambda i_{X}\widetilde{\Omega}%
-i_{X}\delta\widetilde{\Omega}$.

We deduce the following equations:

$a/f\cdot i_{E}\left[  i_{1}\widetilde{\Omega}\right]  $ $=i_{E}\left[
i_{1}(\delta\widetilde{\beta}+\alpha\Lambda\widetilde{\beta})\right]  $,

$b/g\cdot i_{E}\left[  i_{1}\widetilde{\Omega}\right]  =i_{E}\left[
i_{1}((\delta1)\Lambda\alpha-i_{E}\delta\widetilde{\Omega})\right]  ,$

$c/h_{X}\cdot i_{E}\left[  i_{1}\widetilde{\Omega}\right]  $ $=i_{E}\left[
i_{1}(\alpha\Lambda i_{X}\widetilde{\Omega}-i_{X}\delta\widetilde{\Omega
})\right]  $.

As $i_{E}\left[  i_{1}\widetilde{\Omega}\right]  =1$, we verify that the
unique solutions are: $f=1+\alpha(1);$ $g=\alpha(E)$ and $h_{X}=\alpha
(X)=\beta\left[  X,E\right]  $ for any $X\in Ker\beta$.
\end{proof}

\begin{theorem}
We have
\[
\delta\widetilde{\Omega}=-\alpha\Lambda\widetilde{\Omega}\text{..}%
\]

\end{theorem}

\begin{proof}
We recall that%
\begin{align*}
i_{1}\widetilde{\Omega}  &  =\widetilde{\beta};\\
i_{E}\widetilde{\Omega}  &  =-\delta1;\\
i_{1}\delta\widetilde{\Omega}  &  =\widetilde{\Omega}-\delta\widetilde{\beta
}\text{.}%
\end{align*}

For any $\varphi\in\mathcal{D}(M)$ with $\varphi=\varphi(1)+X+\widetilde
{\beta}(\varphi)\cdot E,(X\in Ker\beta)$, we have%
\begin{align*}
i_{\varphi}(\alpha\Lambda\widetilde{\Omega}+\delta\widetilde{\Omega}) &
=\alpha(\varphi)\cdot\widetilde{\Omega}-\alpha\Lambda i_{\varphi}%
\widetilde{\Omega}+i_{\varphi}\delta\widetilde{\Omega}\\
&  =\left[  \varphi(1)\cdot\alpha(1)+\alpha(X)+\widetilde{\beta}(\varphi
)\cdot\alpha(E)\right]  \cdot\widetilde{\Omega}\\
&  -\alpha\Lambda\left[  \varphi(1)\cdot i_{1}\widetilde{\Omega}%
+i_{X}\widetilde{\Omega}+\widetilde{\beta}(\varphi)\cdot i_{E}\widetilde
{\Omega}\right]  \\
&  +\varphi(1)\cdot i_{1}\delta\widetilde{\Omega}+i_{X}\delta\widetilde
{\Omega}+\widetilde{\beta}(\varphi)\cdot i_{E}\delta\widetilde{\Omega}\text{.}%
\end{align*}

We get
\begin{align*}
i_{\varphi}(\alpha\Lambda\widetilde{\Omega}+\delta\widetilde{\Omega})  &
=\left[  \varphi(1)\cdot\alpha(1)+\alpha(X)+\widetilde{\beta}(\varphi
)\cdot\alpha(E)\right]  \cdot\widetilde{\Omega}\\
&  -\alpha\Lambda\left[  \varphi(1)\cdot\widetilde{\beta}+i_{X}\widetilde
{\Omega}-\widetilde{\beta}(\varphi)\cdot\delta1\right] \\
&  +\varphi(1)\cdot(\widetilde{\Omega}-\delta\widetilde{\beta})+i_{X}%
\delta\widetilde{\Omega}+\widetilde{\beta}(\varphi)\cdot i_{E}\delta
\widetilde{\Omega}\text{.}%
\end{align*}

We also have%
\begin{align*}
i_{\varphi}(\alpha\Lambda\widetilde{\Omega}+\delta\widetilde{\Omega})  &
=\varphi(1)\cdot(\left[  1+\alpha(1)\right]  \cdot\widetilde{\Omega}%
-\delta\widetilde{\beta}-\alpha\Lambda\widetilde{\beta})\\
&  +\left[  \alpha(X)\cdot\widetilde{\Omega}-\alpha\Lambda i_{X}%
\widetilde{\Omega}+i_{X}\delta\widetilde{\Omega}\right] \\
&  +\widetilde{\beta}(\varphi)\cdot\left[  \alpha(E)\cdot\widetilde{\Omega
}-(\delta1)\Lambda\alpha+i_{E}\delta\widetilde{\Omega}\right]  \text{.}%
\end{align*}

Proposition $17$, above, implies that%
\[
i_{\varphi}(\alpha\Lambda\widetilde{\Omega}+\delta\widetilde{\Omega
})=0\text{.}%
\]

As
\[
i_{\varphi}(\alpha\Lambda\widetilde{\Omega}+\delta\widetilde{\Omega})=0,
\]
then, for any $\varphi,\psi,\eta\in\mathcal{D}(M)$ we have
\[
(\alpha\Lambda\widetilde{\Omega}+\delta\widetilde{\Omega})(\varphi,\psi
,\eta)=0\text{.}%
\]

We conclude that
\[
\delta\widetilde{\Omega}=-\alpha\Lambda\widetilde{\Omega}\text{.}%
\]

That ends the proof.
\end{proof}

\begin{corollary}
\bigskip If $M$ is a contact manifold, then $\mathcal{D}(M)$ admits a
symplectic Lie-Rinehart-Jacobi algebra structure.
\end{corollary}

In this paper we showed that $\mathcal{D}(M)$ admits a symplectic
Lie-Rinehart-Jacobi algebra structure if and only if $M$ \ is a contact
manifold. Thus a contact structure on a manifold $M$ is due to the existence
of a $C^{\infty}(M)$-linear form%
\[
\alpha:\mathcal{D}(M)\longrightarrow C^{\infty}(M)
\]
and a nondegenerate skew-symmetric bilinear form%
\[
\omega:\mathcal{D}(M)\times\mathcal{D}(M)\longrightarrow C^{\infty}(M)
\]
such that

\begin{enumerate}
\item $\delta\alpha=(\delta1)\Lambda\alpha$;

\item $\delta\omega=-\alpha\Lambda\omega$.
\end{enumerate}

If $\alpha(1)\neq-1$, we will say that $M$ is an exact contact manifold and if
$\alpha(1)=-1$, we will say that $M$ is a nonexact contact manifold.

Thus the parallelism between locally conformal symplectic manifolds and
contact manifolds is obvious: a locally conformal symplectic structure on a
manifold $M$ is due to the existence of a symplectic Lie-Rinehart-Jacobi
algebra structure on $\mathfrak{X}(M)$ whereas a contact structure on a
manifold $M$ is due to the existence of a symplectic Lie-Rinehart-Jacobi
algebra structure on $\mathcal{D}(M)$.


\begin{thebibliography}{9}                                                                                                %


\bibitem {bou1}BOURBAKI \ N., Alg\`{e}bre, chapitres 1 \`{a} 3, Hermann, Paris 1970;

\bibitem {car}Cartan, H., Eilenberg, S., Homological Algebra, Princeton Univ.
Press, 1956, (MR\ 17, 1040.)

\bibitem {god}GODBILLON C., G\'{e}om\'{e}trie diff\'{e}rentielle et
M\'{e}canique analytique, Collections M\'{e}thodes, Hermann, Paris 1969;

\bibitem {lic1}LICHNEROWICZ A., Les vari\'{e}t\'{e}s de Jacobi et leurs
alg\`{e}bres de Lie associ\'{e}es, J. Math. pures et appl., 57, 1978, n$%
{{}^\circ}%
$ $4$, p. 453 \`{a} 488;

\bibitem {oka1}OKASSA E., Alg\`{e}bres de Jacobi et Alg\`{e}bres de
Lie-Rinehart-Jacobi, J.of Pure and Applied Algebra, vol. 208, n$%
{{}^\circ}%
3,$ (2007), 1071-1089;

\bibitem {oka2}OKASSA E., On Lie-Rinehart-Jacobi algebras, Journal of Algebra
and its Aplications, Vol. 7, N$%
{{}^\circ}%
6$, (2008), 749-772;

\bibitem {oka3}OKASSA E, Symplectic Lie-Rinehart-Jacobi algebras and contact
manifolds, Canad. Math. Bull. Vol. 54 (4), 2011 pp. 716-725;

\bibitem {rin}RINEHART G., Differential forms for general commutative
algebras, Trans. Amer. Math. Soc. 108 (1963), 195-222;

\bibitem {vai1}VAISMAN I., Locally conformal symplectic manifolds, Internat.
J. Math. \& Math. Sci. Vol.8 n${{}^{\circ}}3$ (1985), 521-536;
\end{thebibliography}
\end{document}